\documentclass[12pt]{article}
\usepackage[]{amsmath,amssymb}
\usepackage{amscd}
\usepackage{latexsym}
\usepackage{cite}
\usepackage{amsthm}

\newtheorem{definition}{Definition}[section]
\newtheorem{theorem}[definition]{Theorem}
\newtheorem{lemma}[definition]{Lemma}

\newtheorem{example}[definition]{Example}

\newtheorem{problem}[definition]{Problem}

\begin{document}
\title{\bf  
A poset $\Phi_n$ whose maximal chains are 
in bijection with the $n \times n$ alternating sign matrices
}
%Tatsuro Ito\footnote{Supported in part by JSPS grant
%18340022.} $\;$   and
\author{
Paul Terwilliger 
}
\date{}
%\footnote{This author gratefully acknowledges 
%support from the FY2007 JSPS Invitation Fellowship Program
%for Reseach in Japan (Long-Term), grant L-07512.}
%}
%\date{}
%to get date printout, comment out above line

\maketitle
\begin{abstract}
For an integer $n\geq 1$, we display a poset $\Phi_n$
whose maximal chains are in bijection with the $n\times n$
alternating sign matrices. The Hasse diagram $\widehat \Phi_n$ 
is obtained from the $n$-cube by adding some edges.
We show that the dihedral
group $D_{2n}$ acts on $\widehat \Phi_n$
as a group of automorphisms.


 \end{abstract}
\section{Introduction}

\noindent  We will be discussing a type of square matrix called
an alternating sign matrix. These matrices were introduced in
\cite{mrr}, and subsequently linked
to many other topics in Combinatorics;
see \cite{bressoud} for an overview.
In 
\cite{lasc}
 the alternating sign matrices are linked to
partially ordered sets in the following way.
For $n\geq 1$ the set of $n \times n $ alternating sign matrices becomes
a distributive lattice,
which is the MacNeille completion of
the Bruhat order on the symmetric group $S_n$.
For more discussion of this see 
\cite{bru}, 
\cite[p.~598]{striker}, 
\cite[p.~2]{striker2}.
In the present paper we link alternating sign matrices to
partially ordered sets in a different way. 
For $n\geq 1$ we display a partially ordered set $\Phi_n$, whose maximal
chains are in bijection with the $n \times n$ alternating
sign matrices. We will also discuss the symmetries of $\Phi_n$.
Before describing our results in more detail, we recall a few terms.
\medskip

\noindent 
Consider a finite poset, with vertex set
$X$ and partial order $\leq $.
For vertices $x, y$ write $x<y$ whenever
$x\leq y$ and $x \not=y$.
A vertex $x$ is {\it maximal} (resp. {\it minimal})
whenever there does not exist a vertex $y$ such that $x < y$ (resp. $y < x$).
For vertices $x,y$ we say that $y$ {\it covers} $x$ whenever
$x < y$ and there does not exist a vertex $z$ such that
$x < z < y$.
For an integer $r\geq 0$, a {\it chain of length $r$} 
is a sequence of vertices $x_0 < x_1< \cdots < x_r$.
This chain is called {\it maximal} whenever (i) $x_r$ is maximal;
(ii) $x_0$ is minimal; (iii) 
$x_i$ covers $x_{i-1}$ for $1 \leq i \leq r$.
The 
{\it Hasse diagram} $\widehat X$ is an undirected graph
with vertex set $X$;
vertices $x,y$ are adjacent in $\widehat X$  whenever one of
$x,y$ covers the other one. 
\medskip

\noindent
For the rest of this paper, fix an integer $n\geq 1$. 
An $n \times n$ matrix is said to
have order $n$.
\medskip

\noindent 
We recall the $n$-cube.
This is an undirected graph, whose 
 vertex set consists of the sequences $(a_1, a_2, \ldots, a_n)$
such that $a_i \in \lbrace 0,1\rbrace$ for $1 \leq i \leq n$.
Vertices $x,y$ of the $n$-cube are adjacent whenever they
differ in exactly one coordinate. The $n$-cube is often called
a {\it hypercube}, or a {\it binary Hamming graph}.
\medskip

\noindent
We now describe our results in more detail.
Consider the poset whose vertex set
consists of the subsets of  
$\lbrace 1,2,\ldots, n\rbrace$; the partial order is $\subseteq $.
For this poset the Hasse diagram is isomorphic to
the $n$-cube. For this poset,
 the maximal
chains are in bijection with the 
permutation matrices of order $n$ \cite[p~142]{bruIntro}. 
For this poset, we
 augment the partial order 
by adding some edges to the Hasse diagram; the resulting
poset is denoted 
 $\Phi_n$.
We show that
the maximal chains in $\Phi_n$ are in bijection with the
alternating sign matrices of order $n$.
We show that the dihedral
group $D_{2n}$ acts on $\widehat \Phi_n$
as a group of automorphisms.

\section{Alternating sign matrices}

In this section we give some background concerning alternating
sign matrices.

\begin{definition}
\label{def:con}
\rm A sequence $(\sigma_0, \sigma_1, \ldots, \sigma_n)$ is
{\it constrained} whenever
\begin{enumerate}
\item[\rm (i)] $\sigma_i \in \lbrace  0,1\rbrace$ for $0 \leq i \leq n$;
\item[\rm (ii)] $\sigma_0=0$ and $\sigma_n=1$.
\end{enumerate}
Let ${\rm Con}_n$ denote the set of constrained sequences.
\end{definition}

\noindent  Given a constrained sequence
$(\sigma_0, \sigma_1, \ldots, 
\sigma_n)$ define
$\alpha_i= \sigma_i - \sigma_{i-1}$ for $1 \leq i \leq n$.
For example, if
$(\sigma_0, \sigma_1, \ldots,  \sigma_n)=
 (0,0,1,1,1,0,1,1,0,0,1)$ then the sequence
$(\alpha_1, \alpha_2, \ldots, \alpha_n)$ is
\begin{equation*}
(0,1,0,0,-1,1,0,-1,0,1).
\end{equation*}
To describe sequences of this sort, we make a definition.

\begin{definition}
\label{def:alt}
\rm
A sequence
$(\alpha_1, \alpha_2, \ldots, \alpha_n)$
is {\it alternating} whenever 
\begin{enumerate}
\item[\rm (i)]
$\alpha_i \in \lbrace 1,0,-1\rbrace$ for $1 \leq i \leq n$;
\item[\rm (ii)] the nonzero coordinates among $(\alpha_1, \alpha_2,
\ldots, \alpha_n)$
form the pattern
$1,-1,1,-1,\ldots, -1, 1$. 
\end{enumerate}
Let ${\rm Alt}_n$ denote the set of alternating sequences.
\end{definition}

\begin{example}\rm For $n=4$ the alternating sequences are
\begin{eqnarray*}
&&(1,0,0,0),          \qquad \qquad             (0,1,-1,1), \\
&&(0,1,0,0),             \qquad \qquad         (1,0,-1,1), \\ 
&&(0,0,1,0),    \qquad \qquad (1,-1,0,1),\\
&&(0,0,0,1), \qquad \qquad (1,-1,1,0).
\end{eqnarray*}
\end{example}

\begin{lemma} 
\label{lem:atos}
We give a bijection
${\rm Con}_n \to 
 {\rm Alt}_n$.
For a constrained sequence
$(\sigma_0, \sigma_1, \ldots, \sigma_n)$
the corresponding alternating sequence
 $(\alpha_1, \alpha_2, \ldots, \alpha_n)$ satisfies
$\alpha_i = \sigma_i - \sigma_{i-1}$ for $1 \leq i \leq n$.
The inverse bijection ${\rm Alt}_n \to 
 {\rm Con}_n$ is described as follows.
For an alternating sequence $(\alpha_1, \alpha_2, \ldots, \alpha_n)$
the corresponding constrained sequence 
$(\sigma_0, \sigma_1, \ldots, \sigma_n)$ satisfies
$\sigma_i = \alpha_1+\alpha_2+\cdots + \alpha_i$ for $0 \leq i \leq n$.
\end{lemma}
\begin{proof} Use Definitions
\ref{def:con},
\ref{def:alt}.
\end{proof}

\begin{definition}
\rm
(See 
\cite{mrr}.)
An {\it alternating sign matrix} (or {\it ASM})
is a square matrix such that each row and column is
alternating.
\end{definition}

\begin{example} Any permutation matrix is an ASM.
\end{example}

\section{The poset $\Phi_n$}

In this section we define the poset $\Phi_n$, and explain
how it is related to alternating sign matrices.

\begin{definition}
\label{def:Pn}
\rm 
We define a poset $\Phi_n$ as follows.
The vertex set consists of the sequences
$(a_1, a_2, \ldots, a_n)$ such
that $a_i \in \lbrace 0,1\rbrace$ for $1 \leq i \leq n$.
By the {\it rank}
of a vertex we mean the number of nonzero 
coordinates.
The partial order $\leq $ is defined as follows.
For vertices $x,y$ let $y$ cover $x$ with respect to $\leq$ 
whenever $y-x$ is alternating; in this case the rank of $y$
is one more than the rank of $x$.
\end{definition}

\noindent We comment on the poset $\Phi_n$.
The Hasse diagram $\widehat \Phi_n$ is obtained from
the $n$-cube by adding some edges.
The poset $\Phi_n$ has a unique mininimal element 
${\bf 0}=(0,0,\ldots, 0)$ and
a unique maximal element
${\bf 1}=(1,1,\ldots, 1)$.
A chain in $\Phi_n$ is maximal if and only if it  has length $n$.
For a maximal chain $x_0<x_1<\cdots<x_n$
the vertex $x_i$ has rank $i$ for $0 \leq i \leq n$.
In particular $x_0=\bf 0$ and $x_n=\bf 1$.

\begin{theorem}
\label{prop:asmbij}
The alternating sign matrices of order $n$ are in bijection with
the maximal chains of the poset $\Phi_n$.
A bijection is described as follows.
Given a maximal chain 
$x_0<x_1<\cdots< x_n$ in $\Phi_n$, the corresponding ASM has
row $i$ equal to $x_{i}-x_{i-1}$ $(1 \leq i \leq n)$.
The inverse bijection is described as follows.
Given an ASM of order $n$, let
$x_i$ denote the sum of rows $1,2,\ldots, i$ $(0 \leq i \leq n)$. 
 Then $x_0 < x_1 < \cdots < x_n$ is the corresponding
 maximal chain in $\Phi_n$.
\end{theorem}
\begin{proof} 
Let ${\rm Chain}_n$ denote the 
set of maximal chains in $\Phi_n$.
Let ${\rm ASM}_n$ denote the set of 
alternating sign matrices of order $n$.
We give a function ${\rm Chain}_n \to {\rm ASM}_n$.
Let 
$x_0<x_1<\cdots< x_n$  denote a maximal chain in $\Phi_n$.
For $1 \leq i \leq n$ the vertex 
$x_i$ covers $x_{i-1}$, so
$x_i-x_{i-1}$ is alternating. Define an $n \times n $ matrix
$A$ with row $i$ equal to $x_i -x_{i-1}$ for $1 \leq i \leq n$.
We show that $A \in {\rm ASM}_n$. By construction each row of $A$
is alternating. We now show that each column of $A$ is alternating.
For $1 \leq j \leq n$ consider column $j$ of $A$.
For this column let $\alpha_i$ denote coordinate $i$ $(1 \leq i \leq n)$.
 For $0 \leq i \leq n$ let $\sigma_i$ denote coordinate $j$
of $x_i$. By construction
$\sigma_i \in \lbrace 0,1\rbrace$ for $0 \leq i \leq n$.
We have $\sigma_0=0$ since $x_0={\bf 0}$, 
and $\sigma_n=1$ since
$x_n={\bf 1}$. By these comments the sequence
$(\sigma_0, \sigma_1,\ldots, \sigma_n)$ is constrained.
By construction $\alpha_i= \sigma_i-\sigma_{i-1}$ for $1 \leq i \leq n$.
By this and
Lemma \ref{lem:atos} the
sequence $(\alpha_1, \alpha_2, \ldots, \alpha_n)$ is alternating.
We have shown that each column of $A$ is alternating.
By the above comments $A\in {\rm ASM}_n$.
We just gave a function ${\rm Chain}_n \to {\rm ASM}_n$.
Next we give a function ${\rm ASM}_n\to {\rm Chain}_n$.
Let $A$ denote an ASM of order $n$.
For $0 \leq i \leq n$
let $x_i$ denote the sum of its rows $1,2,\ldots, i$.
By Lemma \ref{lem:atos} each coordinate of $x_i$ is $0$ or $1$.
So $x_i$ is a vertex of $\Phi_n$.
By construction $x_0=\bf 0$ and $x_n=\bf 1$.
For $1 \leq i \leq n$,  $x_i-x_{i-1}$ is equal to
row $i$ of $A$, and
is therefore alternating. So $x_i$ covers $x_{i-1}$.
By these comments $x_0< x_1< \cdots < x_n$ is a maximal
chain in $\Phi_n$.
We have given a function ${\rm ASM}_n \to {\rm Chain}_n$.
Consider the above functions ${\rm Chain}_n\to {\rm ASM}_n$ 
and ${\rm ASM}_n\to {\rm Chain}_n$.
These functions are inverses, and hence bijections.
\end{proof}

\section{The graph $\widehat \Phi_n$}
In the previous section we showed how
the 
poset $\Phi_n$ is related to the alternating sign matrices of
order $n$. As we investigate $\Phi_n$, it is sometimes 
convenient to work with the 
Hasse diagram $\widehat \Phi_n$.
Our next goal is to show 
that the dihedral group $D_{2n}$
acts on the graph $\widehat \Phi_n$ as a group of automorphisms.
\begin{definition}
\label{def:theta}
\rm
Define a map $\rho$ on $\Phi_n$
that sends
each vertex 
$(a_1, a_2, \ldots, a_n)\mapsto 
(a_n,a_1, a_2, \ldots, a_{n-1})$.
The map $\rho$ has order $n$.
Define a map $\xi $ on $\Phi_n$ that sends
each vertex 
$(a_1, a_2, \ldots, a_n)\mapsto 
(\overline {a_1}, a_2, \ldots, a_n)$,
where $\overline x$ means $1-x$.
The map $\xi$ has order 2.
Define the composition
$\theta= \xi \circ \rho$.
\end{definition}

\begin{lemma} 
\label{lem:theta}
The map $\theta$ sends each vertex
$(a_1, a_2, \ldots, a_n)\mapsto 
(\overline {a_n}, a_1, \ldots, a_{n-1})$.
Moreover $\theta^{-1}$ sends
$(a_1, a_2, \ldots, a_n)\mapsto 
(a_2, a_3, \ldots, a_n, \overline {a_1})$.
\end{lemma}
\begin{proof} Use Definition
\ref{def:theta}.
\end{proof}

\begin{example}\rm
For $n=2$ the map $\theta$ sends
\begin{align*}
& 
(0,0)\mapsto 
(1,0)\mapsto 
(1,1)\mapsto 
(0,1)\mapsto 
(0,0).
\end{align*}
\end{example}

\begin{example} 
For $n=3$ the map $\theta$ sends
\begin{align*}
& 
(0,0,0)\mapsto 
(1,0,0)\mapsto 
(1,1,0)\mapsto 
(1,1,1)\mapsto
(0,1,1)\mapsto 
(0,0,1) \mapsto 
(0,0,0)
\\
& (0,1,0)\mapsto (1,0,1)\mapsto (0,1,0).
\end{align*}
\end{example}

\begin{example}\rm For $n=4$ the map $\theta$ sends
\begin{align*}
(0,0,0,0)\mapsto 
(1,0,0,0)&\mapsto 
(1,1,0,0)\mapsto 
(1,1,1,0)\mapsto 
(1,1,1,1)
\mapsto (0,1,1,1)
\\
&
\mapsto 
(0,0,1,1)\mapsto 
(0,0,0,1)\mapsto 
(0,0,0,0)
\\
(1,0,0,1)\mapsto 
(0,1,0,0)&\mapsto 
(1,0,1,0)\mapsto 
(1,1,0,1)\mapsto 
(0,1,1,0)
\mapsto 
(1,0,1,1)
\\
&
\mapsto 
(0,1,0,1)\mapsto 
(0,0,1,0)\mapsto 
(1,0,0,1).
\end{align*}
\end{example}

\begin{example}\rm
For $n=5$ the map $\theta$ sends
\begin{align*}
(0,0,0,0,0) &\mapsto 
(1,0,0,0,0)\mapsto 
(1,1,0,0,0)\mapsto 
(1,1,1,0,0)\mapsto 
(1,1,1,1,0) \mapsto
(1,1,1,1,1) 
\\
&
\mapsto 
(0,1,1,1,1) \mapsto
(0,0,1,1,1) \mapsto
(0,0,0,1,1) \mapsto
(0,0,0,0,1) \mapsto
(0,0,0,0,0)
\\
(1,0,0,0,1) &\mapsto
(0,1,0,0,0) \mapsto
(1,0,1,0,0) \mapsto
(1,1,0,1,0) \mapsto
(1,1,1,0,1) \mapsto
(0,1,1,1,0)
\\
&\mapsto
(1,0,1,1,1) \mapsto
(0,1,0,1,1) \mapsto
(0,0,1,0,1) \mapsto
(0,0,0,1,0) \mapsto
(1,0,0,0,1)
\\
(0,0,1,0,0) &\mapsto 
(1,0,0,1,0) \mapsto 
(1,1,0,0,1) \mapsto 
(0,1,1,0,0) \mapsto 
(1,0,1,1,0) \mapsto 
(1,1,0,1,1)
\\
&\mapsto 
(0,1,1,0,1) \mapsto 
(0,0,1,1,0) \mapsto 
(1,0,0,1,1) \mapsto 
(0,1,0,0,1) \mapsto 
(0,0,1,0,0)
\\
(1,0,1,0,1) &\mapsto (0,1,0,1,0) \mapsto (1,0,1,0,1).
\end{align*}
\end{example}

\begin{lemma}
The map $\theta^n$ sends each vertex
$(a_1, a_2, \ldots, a_n)\mapsto 
(\overline {a_1}, \overline {a_2}, \ldots, \overline {a_n})$.
Moreover $\theta$
has order $2n$.
\end{lemma}
\begin{proof} Use Lemma
\ref{lem:theta}.
\end{proof}

\begin{lemma}
\label{lem:thetaaut}
The map
$\theta$ from 
Definition
\ref{def:theta}
is an automorphism
of the graph  $\widehat \Phi_n$.
\end{lemma}
\begin{proof} By construction $\theta$ permutes the vertices
of $\widehat \Phi_n$. One checks that this permutation
respects adjacency in $\widehat \Phi_n$.
\end{proof}

\begin{definition}
\label{def:tau}
\rm
Define a map $\tau$ on $\Phi_n$ that sends
each vertex $(a_1, a_2, \ldots, a_n)\mapsto 
(a_n, \ldots, a_2, a_1)$.
Note that $\tau$ has order 2.
\end{definition}

\begin{lemma}
\label{lem:tauaut}
The above map $\tau$ is an automorphism
of the graph $\widehat \Phi_n$.
\end{lemma}
\begin{proof} This is routinely checked.
\end{proof}

\begin{lemma}
\label{lem:tautheta} We have
$\theta \circ \tau = \tau \circ \theta^{-1}$.
\end{lemma}
\begin{proof}
Use Definitions
\ref{def:theta},
\ref{def:tau}.
\end{proof}

\begin{theorem}
The above maps $\theta, \tau$ induce an action of $D_{2n}$
on the graph $\widehat \Phi_n$ as a group of automorphisms.
\end{theorem}
\begin{proof} By Lemmas
\ref{lem:thetaaut},
\ref{lem:tauaut},
\ref{lem:tautheta}
and the construction.
\end{proof}

\bigskip

\noindent Paul Terwilliger \hfil\break
\noindent Department of Mathematics \hfil\break
\noindent University of Wisconsin \hfil\break
\noindent 480 Lincoln Drive \hfil\break
\noindent Madison, WI 53706-1388 USA \hfil\break
\noindent email: {\tt terwilli@math.wisc.edu }\hfil\break

\end{document}